\documentclass[a4paper,12pt]{article}
\usepackage{amssymb,latexsym,amsmath,enumerate,verbatim,amsfonts,amsthm}
\usepackage[active]{srcltx}

\newtheorem{lemma}{Lemma}[section]

\newtheorem{theorem}{Theorem}[section]
\newtheorem{corollary}{Corollary}[section]

\newtheorem{remark}{Remark}[section]
\newtheorem{example}{Example}[section]

\begin{document}
\title{{\sf Convergence of the Lasserre Hierarchy of SDP Relaxations for
Convex Polynomial Programs without Compactness\thanks{ Research was partially supported by a grant from the Australian Research Council}} }

\author{\sc V. Jeyakumar\thanks{
Department of Applied Mathematics, University of New South Wales,
Sydney 2052, Australia. E-mail: v.jeyakumar@unsw.edu.au}, T. S. Ph\d{a}m\thanks{ Department of Mathematics, University of Dalat, 1, Phu Dong Thien Vuong, Dalat, Vietnam. Email: sonpt@dlu.edu.vn. This work was carried out while the author was a visitor to the School of Mathematics, University
of New South Wales, Sydney, Australia. Research was partially funded by the Vietnam National Foundation for Science and Technology Development (NAFOSTED)} \ and \ G. Li\thanks{Department of Applied
Mathematics, University of New South Wales, Sydney 2052, Australia.
E-mail: g.li@unsw.edu.au} }

\date{April 12, 2013}

\maketitle
\begin{abstract}
The Lasserre hierarchy of semidefinite programming (SDP) relaxations is an effective scheme for
finding computationally feasible SDP approximations of polynomial optimization over \textit{compact} semi-algebraic sets.
In this paper, we show that, for convex polynomial optimization, the Lasserre hierarchy with a slightly extended quadratic module \textit{always converges asymptotically} even in the face of non-compact semi-algebraic feasible sets. We do this by exploiting a coercivity property of convex polynomials that are bounded below. We further establish that the positive definiteness of the Hessian of the associated Lagrangian at a saddle-point (rather than the objective function at each minimizer) guarantees \textit{finite convergence} of the hierarchy. We obtain finite convergence by first establishing a new sum-of-squares polynomial representation of convex polynomials over convex semi-algebraic sets under a saddle-point condition. We finally prove that the existence of a saddle-point of the Lagrangian for a convex polynomial program is also necessary for the hierarchy to have finite convergence.


\bigskip
{\bf AMS subject class:} 90C60, 90C56, 90C26
\end{abstract}

\section{Introduction} \label{SectionIntroduction}

\noindent
\noindent When it comes to polynomial optimization over compact semi-algebraic feasible sets, Lasserre's hierarchy of
semidefinite programming (SDP) relaxations \cite{Lasserre2009} is an effective scheme for solving
polynomial optimization problems via computationally feasible approximations. The hierarchy has asymptotic convergence in the sense that the sequence of optimal values of the SDP relaxations converges to the optimal value of the original
problem \cite{Lasserre2009-1, Lasserre2009} under mild assumptions. It has finite convergence for convex polynomial optimization over compact semi-algebraic sets whenever the Hessian of the convex polynomial is positive definite at
each minimizer \cite{Klerk2011, Lasserre2009-1, Lasserre2009}, requiring strict
convexity of the convex polynomial (see Lemma 2.1 in Section 2). The proofs of these convergence hold in the compact case of the semi-algebraic feasible sets and they
rely on the powerful sum-of-squares polynomial representation of positive polynomials
over compact semi-algebraic sets from real algebraic geometry \cite{Put,sch}.
%
%
\medskip

\noindent The purpose of this paper is to show that, in the case of a \textit{non-compact} semi-algebraic feasible set of a convex polynomial program, an extended {\it quadratic module}, generated in terms of both the convex polynomial objective function and the polynomials associated with the semi-algebraic set, leads to a converging hierarchy of semidefinite programming (SDP) relaxations.

%
%
\medskip

\noindent
{\bf \textit{Main Contributions}}\\
\noindent
We establish that the Lasserre hierarchy of SDP approximations with the extended quadratic module {\em always converges asymptotically} for convex
polynomial programs without any compactness assumptions on the feasible sets. We also show that the positive definiteness of the Hessian of the Lagrangian at a saddle-point guarantees finite convergence of the hierarchy.
\medskip

\noindent
We prove asymptotic convergence of the hierarchy by exploiting a coercivity property of convex polynomials that are bounded below. On the other hand, we derive finite convergence by first proving that a convex polynomial with {\it positive definite Hessian at a single point} is strictly convex and coercive, and then  establishing that the positive definiteness of the Hessian of the Lagrangian at a saddle-point guarantees a sum-of-squares representation of a convex polynomial over a convex (not necessarily compact) semi-algebraic set.
\medskip

\noindent
Moreover, we establish that the existence of a saddle-point of the associated Lagrangian at every
minimizer of the convex problem is necessary for the Lasserre hierarchy to have finite convergence. We give simple numerical examples explaining the assumptions of our theorems.
\medskip

\noindent
{\bf \textit{Significance of our Contributions}}\\
\noindent
The Lasserre hierarchy of SDP approximations with our extended quadratic module is significant for convex polynomial programming
because it not only converges asymptotically without any regularity conditions on the feasible set but also exhibits finite convergence without the standard strict convexity requirement of the objective function. Our conditions for finite convergence are given in terms of positive
definiteness of the associated Lagrangian
function rather than just the objective function (c.f. \cite{Klerk2011, Lasserre2009-1}).
\medskip

\noindent
The significance of our sum-of-squares polynomial representation is that it allows us to construct a hierarchy of
SDP approximations in terms of \textit{quadratic modules rather than
pre-orderings} \cite{Demmel2007, HaHV2008-1, HaHV2009-1} even in the case of convex programs with non-compact feasible sets. Also, our representation
extends the corresponding known representations of convex polynomials over compact feasible sets \cite{Klerk2011, Lasserre2009-1}.
\medskip

%
%
%

\section{Convergence of Lasserre Hierarchy without Compactness} \label{SectionOptimization}

We begin by fixing notation and definitions. Throughout this paper, $\mathbb{R}^n$ denotes the Euclidean space with dimension $n$. The inner product in $\mathbb{R}^n$ is defined by $\langle x,y\rangle := x^Ty$ for all $x,y \in \mathbb{R}^n$. The non-negative orthant of $\mathbb{R}^n$ is denoted by $\mathbb{R}^n_{+}$ and is defined by
$\mathbb{R}^n_{+}:=\{(x_1,\ldots,x_n) \in \mathbb{R}^n \ | \ x_i \ge 0\}$. Denote by $\mathbb{R}[\underline{x}]$ the ring of polynomials in $x := (x_1,x_2, \ldots, x_n)$ with real coefficients.

A symmetric $n \times n$ matrix ${\mathcal A}$ is said to be {\em positive definite},
denoted by ${\mathcal A} \succ 0,$  if  $x^T{\mathcal A}x > 0$ for all $x \in \mathbb{R}^n, x \ne 0.$ The gradient and the Hessian of a real polynomial $f\in \mathbb{R}[\underline{x}]$ at a point $x^*$ are denoted by $\nabla f(x^*)$ and $\nabla^2 f(x^*)$ respectively.
Moreover, for a function $L \colon \mathbb{R}^n \times \mathbb{R}^m \rightarrow \mathbb{R}$, we use $\nabla^2_{xx}L(x,\lambda)$ to denote the second order derivative with respect to the variable $x$.

We say that a real polynomial $f\in \mathbb{R}[\underline{x}]$ is {\em sum of squares} (SOS) if there exist
real polynomials $f_j, j = 1,\ldots, r,$ such that $f=\sum_{j=1}^rf_j^2$. The set of all sum-of-squares
real polynomials is denoted by $\Sigma^2.$ An important property of the sum-of-squares polynomials is
that checking a polynomial is sum of squares or not is equivalent to solving a semidefinite linear programming problem. For details see \cite{Lasserre2009, Marshall2008, Parrilo2000}.

Recall that a {\em quadratic module generated by polynomials} $-g_1, \ldots, -g_m \in {\Bbb R}[\underline{x}]$ is defined as
$$\mathbf{M}(-g_1, \ldots, -g_m) := \{ \sigma_0 - \sigma_1 g_1 - \cdots - \sigma_m g_m \ | \ \sigma_i \in \Sigma^2, i = 0, 1, \ldots, m\}.$$
It is a subset of polynomials that are non-negative on the set
$\{x \in {\Bbb R}^n \ | \ g_i(x) \le 0, i  =1, \ldots, m\}$ and possess a very nice certificate for this property.

The quadratic module $\mathbf{M}(-g_1, \ldots, -g_m)$ is called {\em Archimedean} \cite{Marshall2008,Schweighofer2006} if there exists
$p \in \mathbf{M}(-g_1, \ldots, -g_m)$ such that  $\{x: p(x) \ge 0\}$ is compact. When the quadratic module $\mathbf{M}(-g_1, \ldots, -g_m)$ is compact, we have the
following important characterization of positivity of a polynomial over a semialgebraic set.

 \begin{lemma}{\bf (Putinar positivstellensatz) \cite{Put}}\label{Putinar}
 Let $f,g_j$, $j=1,\ldots,m$, be real polynomials with $K:=\{x:g_j(x) \le 0,j=1,\ldots,m\} \neq \emptyset$. Suppose that $f(x)>0$ for all $x \in K$ and $M(-g_1,\ldots,-g_m)$ is Archimedean. Then,
$f \in M(-g_1,\ldots,-g_m)$.
\end{lemma}

In this section we examine the Lasserre SDP relaxation scheme to the following convex programming problem with polynomials:
\begin{eqnarray} \label{OptimizationProblem}
f^* &:= &\min_{x\in \mathbb{R}^n}\{ f(x) \ | \ g_i(x) \le 0, \  i=1,2,\ldots,m\},
\end{eqnarray}
where $f, g_1, \ldots, g_m$ are convex polynomials on  $\mathbb{R}^n$ and
$$K := \{x\in\mathbb{R}^n \ | \ g_1(x) \le 0, \ldots, g_m(x) \le 0 \} \neq \emptyset.$$

Let $c \in {\Bbb R}$ be such that $c > f(x^0)$ for some $x^0 \in K.$ For each integer $k,$ we define the truncated quadratic module $\mathbf{M}_k$ generated by the polynomials $c - f$ and $-g_1, \ldots, - g_m$ as
\begin{eqnarray*}
\mathbf{M}_k \ := \ \{ \sigma_0 - \sum_{i = 1}^m \sigma_i g_i + \sigma (c - f)  & | & \sigma,\sigma_0, \sigma_1,\ldots,\sigma_m \in \Sigma^2 \subset {\Bbb R}[\underline{x}], \\
&&  \deg \, \sigma_0 \le 2k, \, \deg \, \sigma_i g_i \le 2k \textrm{ and } \deg \, \sigma (c - f) \le 2k\}.
\end{eqnarray*}
Consider the following relaxation problem
\begin{equation} \label{SOSRelaxationProblem1}
f_k^* :=  \sup\{\mu \in {\Bbb R} \ | \ f - \mu \in \mathbf{M}_{k}   \}.
\end{equation}

As is well known, the problem of computing the supremum $f^*_k$ can be reduced to a semidefinite program (see \cite{Lasserre2001}, \cite{Lasserre2009}, \cite{Marshall2008}, \cite{Parrilo2000}). Moreover, we can see that
$$f^*_k \le f^*_{k + 1} \le \cdots \le f^*.$$

The following useful coercivity property of a convex polynomial, that is bounded below, allows us to establish that the Lasserre hierarchy of
SDP relaxations of Problem (\ref{OptimizationProblem}) {\it has an  asymptotic  convergence} in the
sense that $f_k^* \uparrow f^*$ as $k \rightarrow \infty$. Recall that a real-valued function $f$ on $\mathbb{R}^n$ is coercive on $\mathbb{R}^n$ whenever $\displaystyle\liminf_{||x||\to \infty}f(x) =+\infty$.

\begin{lemma}[{\bf Coercivity and Convex Polynomials}]\label{LA}
 Let $h \in \mathbb{R}[\underline{x}]$ be a convex polynomial which is bounded below on $\mathbb{R}^n$. Then there exist an orthogonal $n \times n$ matrix $A$ and a coercive polynomial $g \colon \mathbb{R}^l \rightarrow \mathbb{R}$, $1 \le l \le n,$ such that
 \begin{eqnarray*}
h(Ax)=h(A(x_1,\ldots, x_l, \ldots,x_n)^T) = g(x_1,\ldots,x_l), \ \textrm{ for } \ x = (x_1,\ldots, x_l, \ldots,x_n)^T \in \mathbb{R}^n.
\end{eqnarray*}
In particular, $h$ attains its infimum on ${\Bbb R}^n.$
\end{lemma}
\begin{proof} The proof is given in Appendix.
\end{proof}

The following known existence result  of a solution of convex polynomial programs will also be useful for the proof of asymptotic convergence.
\begin{lemma}\label{existence} \cite{Belousov2002}
 Let $f_0,f_1,\ldots,f_m$ be convex polynomials on $\mathbb{R}^n$. Let $C:=\{x\in \mathbb{R}^n : f_i(x) \le 0, i=1,\ldots,m\}$. Suppose that $\inf_{x \in C} f_0(x)>-\infty$. Then, $\displaystyle {\rm argmin}_{x \in C}f_0(x) \neq \emptyset$.
\end{lemma}

\begin{theorem}[{\bf Asymptotic Convergence}] \label{AsymptoticConvergenceTheorem}
For Problem~{\rm (\ref{OptimizationProblem})}, let $x^*$ be a minimizer. Then,
$\displaystyle\lim_{k\to \infty}f^*_k = f^*$.
\end{theorem}
\begin{proof}
{\bf [Positivity of Approximate Lagrangian by Convex Programming Duality]}. Let $\epsilon > 0$.
We first prove that there exists  $\lambda \in \mathbb{R}_+^{m}$ such that
$$f(x) -f(x^*) + \sum_{i=1}^m \lambda_i g_i(x) + \epsilon > 0, \qquad \forall x \in {\Bbb R}^n.$$

Note that, by the assumption, $f-f(x^*) \ge 0$ on $K$, where
$K := \{x\in\mathbb{R}^n \ | \ g_1(x) \le 0, \ldots, g_m(x) \le 0 \}$. Then, $f + \epsilon -f(x^*) > 0$ on $K$.
So, there exists $\delta > 0$ such that $f + \epsilon -f(x^*) > 0$ on $K_\delta$, where
$K_\delta := \{x\in\mathbb{R}^n \ | \ g_1(x) \le \delta , \ldots, g_m(x) \le \delta \}.$
Otherwise, we can find a sequence $\{\delta_k\}\subset \mathbb{R}_+$, $\delta_k\to 0$ and $\{x_k\}\subset \mathbb{R}^n$ such that $g_i(x_k)\le \delta_k$, $i=1,2,\ldots, m$ and $f(x_k)+\epsilon - f(x^*) \le 0$.  Then,
\begin{eqnarray*}
0&\le &\displaystyle\inf_{x,z_1,\ldots,z_m}\{ \sum_{i=1}^m z_i^2
 \ | \ f(x)+ \epsilon -f(x^*)\le 0, \ g_i(x)-z_i \le 0, \ i=1,\ldots,m \}\\
 &\le &\sum_{i=1}^m\delta_k^2=m\delta_k^2\rightarrow 0, \ \ \mbox{as} \ \ k\to \infty.
 \end{eqnarray*}
So, from Lemma \ref{existence} that  there exist
$y^* \in \mathbb{R}^n$ and $z^*=(z_1^*, \ldots, z_m^*)\in\mathbb{R}^{m}$
such that $f(y^*) +\epsilon -f(x^*)\le 0, \ g_i(y^*)-z_i^* \le 0, i=1,\ldots,m,$
and $\sum_{i=1}^m {z_{i}^*}^2=0$. Thus, $f(y^*) +\epsilon -f(x^*)\le 0$ and $g_i(y^*)\le 0$, $i=1,2,\ldots, m$. This is a contradiction.

Now, by Lemma \ref{existence}, $f$ attains its minimizer at $w^* \in K_\delta$ with
$f(w^*) + \epsilon -f(x^*) > 0$. As $g_i(x^*)\le 0 < \delta, i=1,2,\ldots, m$, the Slater condition holds
for the constraints, $g_1(x) \le \delta , \ldots, g_m(x)\le \delta$, and so,
by the convex programming duality \cite{Hiriart1993,Jeya_book,jeya1}, there exist $\lambda_i \ge 0, i=1,2,\ldots,m$ such that,
for all $x\in \mathbb{R}^n$, $f(x) + \sum_{i=1}^m \lambda_i (g_i(x)-\delta) \ge f(w^*)$.
This gives us that, for all $x\in \mathbb{R}^n$,
$f(x) + \sum_{i=1}^m \lambda_i g_i(x) \ge f(w^*) + \sum_{i=1}^m \lambda_i \delta
\ge f(w^*)> f(x^*) -\epsilon$.

\medskip

\noindent {\bf [Asymptotic Representation by Putinar Positivstellensatz]}. Let, for each $x \in \mathbb{R}^n$,
$$h(x):=f(x) -f(x^*) + \sum_{i=1}^m \lambda_i g_i(x) + \epsilon .$$ Then, $h$ is a convex polynomial which is positive on $\mathbb{R}^n$.
Lemma \ref{LA} shows that there exist an orthogonal $n \times n$ matrix $A$ and a coercive polynomial $g \colon \mathbb{R}^l \rightarrow \mathbb{R}$ such that
\begin{eqnarray}\label{eq:00pp}
h(A(x_1,\ldots, x_l, \ldots,x_n)) = g(x_1,\ldots,x_l), \ \ for \ x=(x_1,\ldots, x_l, \ldots,x_n) \in \mathbb{R}^n.
\end{eqnarray}
Let $T=\{x \in \mathbb{R}^n: h(x) \le c-f(x^*) +\epsilon\}$. Then, $T$ is nonempty.
As $g$ is coercive on $\mathbb{R}^l$, it follows from  (\ref{eq:00pp}) that
$$S := \{x \in \mathbb{R}^l \ | \ g(x_1,\ldots,x_l) \le c-f(x^*)+\epsilon \}$$
is a nonempty and compact set. The positivity of $h$ guarantees that $g > 0$ over $\mathbb{R}^l$, and in
particular $g > 0$ over $S.$ Let $p(x)=g(x)-c+f(x^*)-\epsilon$ for all $x \in \mathbb{R}^l$. Then ${\bf M}(-p)$ is Archimedean as $-p \in {\bf M}(-p)$ and $\{x: -p(x) \ge 0\}=S$ is compact. Then, Putinar Positivstellensatz (Lemma \ref{Putinar}) gives us that there exist sum-of-squares polynomials $\sigma_0,\sigma_1$ over $\mathbb{R}^l$ such that
$$g=\sigma_0+\sigma_1 (c-f(x^*)-g+ \epsilon).$$
From (\ref{eq:00pp}),  for each $x = (x_1,\ldots,x_l,x_{l+1},\ldots,x_n) \in \mathbb{R}^n,$ $h(Ax)=g(x_1,\ldots,x_l)$. So,
for each $x = (x_1,\ldots,x_l,x_{l+1},\ldots,x_n) \in \mathbb{R}^n,$
\[
h(Ax)= \sigma_0(x_1,\ldots,x_l)+ \sigma_1(x_1,\ldots,x_l) (c-h(Ax)-f(x^*)+ \epsilon ).
\]
Then, for each $z \in \mathbb{R}^n$,
\[
h(z)=\sigma_0\big((A^{-1}z)_1,\ldots,(A^{-1}z)_l\big) + \sigma_1\big((A^{-1}z)_1,\ldots,(A^{-1}z)_l\big) (c-h(z)-f(x^*)+\epsilon).
\]
Using the definition of $h$, we see that, for each $z \in \mathbb{R}^n$,
\begin{eqnarray*}
&&f(z) -f(x^*) + \sum_{i=1}^m \lambda_i g_i(z) + \epsilon \\
&=& \sigma_0\big((A^{-1}z)_1,\ldots,(A^{-1}z)_l\big) + \sigma_1\big((A^{-1}z)_1,\ldots,(A^{-1}z)_l\big) (c-f(z)-\sum_{i=1}^m \lambda_i g_i(z)).
\end{eqnarray*}
Thus, for each $z \in \mathbb{R}^n$,
\begin{eqnarray}\label{AR}
& & f(z) -f(x^*) + \epsilon \\ \nonumber
&=& \sigma_0\big((A^{-1}z)_1,\ldots,(A^{-1}z)_l\big) + \sigma_1\big((A^{-1}z)_1,\ldots,(A^{-1}z)_l\big) (c-f(z)) \\ \nonumber
&& -\sum_{i=1}^m \bigg(\sigma_1\big((A^{-1}z)_1,\ldots,(A^{-1}z)_l\big)\lambda_i+
\lambda_i\bigg)g_i(z) \ ,
\end{eqnarray}
where $z \mapsto \sigma_i\big((A^{-1}z)_1,\ldots,(A^{-1}z)_l\big)$, $i=0,1$, are sum-of-squares polynomials and $\lambda_i \ge 0$, for $i=1,2,\ldots, m$.

\noindent{\bf [Convergence from Asymptotic Representation]}.
Equation (\ref{AR}) shows that, for each $\epsilon > 0$,
$f  - f^*+\epsilon \in \mathbf{M}(-g_1, \ldots, -g_m, c - f).$
So, there exists $k \in \mathbb{N}$ such that
$f^*-\epsilon \le f_k^*$. This together with the fact that  $f^*_k \le f^*_{k + 1} \le \cdots \le f^*$ gives us that $\displaystyle\lim_{k\to \infty}f^*_k = f^*$.

\end{proof}

\section{Sums of Squares Representations and Finite Convergence} \label{SectionRepresentations}

In this section, we present new representation results for non-negativity of convex polynomials over
convex semi-algebraic sets. For related results, see
\cite{HaHV2010, Helton2010, Marshall2009, Nie2006, Schweighofer2005, Schweighofer2006} and other references therein.



The following Lemma on strict convexity and coercivity of convex polynomials plays a key role in proving the desired representation of convex polynomials and then the finite convergence of the Lasserre hierarchy.

\begin{lemma} [{\bf Hessian Condition for Coercivity and Strict Convexity}] \label{HessianLemma}
Let $f \in \mathbb{R}[\underline{x}] $ be a convex polynomial. If $\nabla^2 f(x_0)\succ 0$ at some point $x_0 \in \mathbb{R}^n$ then $f$ is coercive and strictly convex
on $\mathbb{R}^n$.
\end{lemma}
\begin{proof}
A simple proof is given in the Appendix.
\end{proof}


Let $f$ and $g_1, \ldots, g_m \in \mathbb{R}[\underline{x}]$ be convex polynomials with $K := \{x\in \mathbb{R}^n \ | \  g_i(x) \le 0, i = 1, \ldots, m\} \ne \emptyset.$
Suppose that ${\rm argmin}_{K}f \neq \emptyset$ and that there exists $x^* \in {\rm argmin}_Kf$. Then, convex programming duality \cite{Hiriart1993,Jeya_book,jeya1,jgd} shows that if there exists $x^0 \in {\Bbb R}^n$ such that $g_i(x^0) < 0$ for $i = 1, \ldots, m$, then there exists $\lambda^*\in \mathbb{R}_+^m$ such that $(x^*, \lambda^*)$ is {\em a saddle-point of the Lagrangian function} $L(x,\lambda) := f(x) + \sum_{i = 1}^m\lambda_ig_i(x)$ in the sense that
for each $x \in \mathbb{R}^n$ and for each $\lambda \in \mathbb{R}_+^m$,
$$L(x, \lambda^*)\ge L(x^*, \lambda^*)\ge L(x^*, \lambda).$$

%

%

\begin{theorem} [{\bf Representation of Convex Polynomials}] \label{TheoremRepresentation41}
Let $f$ and $g_1, \ldots, g_m \in \mathbb{R}[\underline{x}]$ be convex polynomials
with $K := \{x\in \mathbb{R}^n \ | \  g_i(x) \le 0, i = 1, \ldots, m\} \ne \emptyset.$
Let $L \colon \mathbb{R}^n \times \mathbb{R}^m_+ \rightarrow \mathbb{R}$ be the Lagrangian function
defined by $L(x,\lambda) := f(x) + \sum_{i = 1}^m\lambda_ig_i(x)$. If the Lagrangian function $L$ has a
saddle-point $(x^*, \lambda^*)\in K\times \mathbb{R}^m_+$ with $\nabla^2_{xx} L(x^*,\lambda^*) \succ 0$,
then, for any $c\in {\Bbb R}$ with $c > f(x^*)$, we have $f - f(x^*) \in \mathbf{M}(-g_1, \ldots, -g_m, c - f).$
\end{theorem}
\begin{proof}
Since $(x^*, \lambda^*)$ is a saddle-point of the Lagrangian function $L$ and $x^*\in K$, it follows that,
for each $x\in \mathbb{R}^n$, $L(x, \lambda^*)\ge L(x^*, \lambda^*)= f(x^*)$ and $x^*$ is a minimizer of
$f$ over $K$. Let
$$h(x) := L(x, \lambda^*) - f(x^*) = f(x) - f(x^*)  + \sum_{i=1}^m \lambda^*_i g_i(x),
\qquad \forall x \in {\Bbb R}^n.$$
Clearly $h$ is a convex polynomial and $h(x) \ge 0$, for all $x\in \mathbb{R}^n$. Moreover, it is easy to check that $h(x^*) = 0 = \inf_{x \in {\Bbb R}^n}h(x);$
in particular, $\nabla h(x^*) = 0.$ By a direct calculation, the Hessian $\nabla^2 h$ of $h$ at $x^*$ is
positive definite. We deduce from Lemma~ \ref{HessianLemma} that the polynomial $h$ is strictly convex and coercive, which implies that $x^*$ is the unique minimizer of $h$ on ${\Bbb R}^n$ and that
$$S : =  \{x \in {\Bbb R}^n \ | \ h(x) \le c-f(x^*)  \}$$
is a nonempty compact set.

We may now apply \cite[Corollary 3.6]{Scheiderer2005} (see also \cite[Example 3.18]{Scheiderer2003}) to conclude that there exist sum-of-squares polynomials $\sigma_0,\sigma_1 \in \Sigma^2$ such that, for each $x\in \mathbb{R}^n$,
\begin{eqnarray*}
h(x) &=&  \sigma_0(x) + \sigma_1(x) (c - f(x^*)-h(x)).
\end{eqnarray*}
This reduces to, for each $x\in \mathbb{R}^n$,
\begin{eqnarray*}
f(x) - f(x^*) & = &  \sigma_0 - \sum_{i=1}^m (\lambda_i^* + \lambda_i^* \sigma_1) g_i(x) + \sigma_1(c - f(x)).
\end{eqnarray*}
Then the conclusion follows.
\end{proof}

\begin{example} [{\bf Importance of positive definite Hessian of $L$ at a saddle-point for representation}] {\rm
Let $p \in {\Bbb R}[x]$ be a convex form (i.e., homogeneous polynomial) on $\mathbb{R}^n$  of degree at least $4$ which is not a sum-of-squares polynomial. See \cite{Blekherman2006} for the existence of such polynomials.

Let $f,g$ be convex polynomials on $\mathbb{R}^{n} \times \mathbb{R}$ defined by $f(x,y) := p(x)$ and
$g(x, y) := y^2 - 1.$ Then, $f$ is not strictly convex. Let $f^* := \min_{x \in K} f(x,y),$ where $K:=\{(x,y) \in \mathbb{R}^{n} \times \mathbb{R} \ | \ g(x,y) \le 0\} = \mathbb{R}^n \times [-1, 1].$

Then $f^* = 0$ because $f(0, 1) = 0$, $\nabla f(0, 1) = 0$ and $f$ is convex.
 Consider the corresponding Lagrangian $L \colon \mathbb{R}^{n+1}\times \mathbb{R}_+ \rightarrow \mathbb{R}$
defined by $L(x,y,\lambda) := f(x,y)+\lambda g(x,y).$ Clearly, $(x^*,y^*,\lambda^*):=(0, 1, 0)$ is a
saddle point of $L$ as $L(x^*,y^*,\lambda)=L(x^*,y^*,\lambda^*) = 0 \le f(x,y) = L(x,y,\lambda^*)$ for all $x \in \mathbb{R}^n$ and $\lambda \in \mathbb{R}_+$. Moreover, as $\nabla^2p(x^*)=\nabla^2p(0)=0$, the Hessian of the Lagrangian function $L$ is not positive definite at the point $(x^*,y^*,\lambda^*).$

We now show that the representation of Theorem \ref{TheoremRepresentation41} fails. To see this, note that the
quadratic module $\mathbf{M}(1 - \|x\|^2) \subset {\Bbb R}[x]$ is Archimedean. So, there exists $c > f(x^*, y^*)=0$
such that  $c - p \in  \mathbf{M}(1 - \|x\|^2).$

On the contrary, suppose that the representation of Theorem \ref{TheoremRepresentation41} holds. Then,
$$f(x,y) = \sigma_0(x,y) -\sigma_1(x,y) g(x,y) + \sigma(x,y) (c - f(x,y)) \mbox{ for all } (x,y) \in \mathbb{R}^{n}\times \mathbb{R},$$
for some sum-of-squares polynomials $\sigma,\sigma_0, \sigma_1$ in the ring ${\Bbb R}[x, y].$ Letting $y = 1$ and noting that $g(x, 1) = 0$, we see that, for all $x \in \mathbb{R}^n$
\begin{eqnarray*}
p(x) = f(x, 1) & = & \sigma_0(x, 1) + \sigma(x, 1) (c - f(x, 1)) \\
&= & \sigma_0(x, 1) + \sigma(x, 1) (c - p(x)).
\end{eqnarray*}
So, $p \in \mathbf{M}(c - p) \subset {\Bbb R}[x].$  Then we have $p \in \mathbf{M}(1 - \|x\|^2).$
By Proposition 4 in De Klerk, Laurent, and Parrilo \cite{Klerk2005}, a form belongs to the quadratic module $\mathbf{M}(1 - \|x\|^2)$ if and only if it is a sum-of-squares polynomial. This contradicts our assumption that the polynomial $p$ is not a sum-of-squares. Thus, the representation fails in this case.
}\end{example}

As an easy application of Theorem 2.1, we obtain the
following representation under the Archimedean assumption. For related results, see
\cite[Theorem 3.4]{Lasserre2009-1} and \cite[Corollary 3.3]{Klerk2011}.

\begin{corollary}[{\bf Representation with Archimedean Condition}] \label{KlerkCorollary}
Let $f, g_1, \ldots, g_m \in \mathbb{R}[\underline{x}]$ be convex polynomials,  and let $K := \{x\in \mathbb{R}^n \ | \  g_i(x) \le 0, i = 1, \ldots, m\} \ne \emptyset.$ Suppose that the following assumptions hold:
\begin{enumerate}
\item[{\rm (i)}] There exists $x^0 \in {\Bbb R}^n \textrm{ such that } g_i(x^0) < 0 \textrm{ for } i = 1, \ldots, m.$
\item[{\rm (ii)}] $\nabla^2_{xx} L(x^*, \lambda^*)\succ 0 $ at a saddle-point $(x^*,\lambda^*) \in \mathbb{R}^n \times \mathbb{R}^m_+$
of the Lagrange function $L$.
\item[{\rm (iii)}] The quadratic module $\mathbf{M}(-g_1 , \ldots, -g_m)$ is Archimedean.
\end{enumerate}
Then, $f - f(x^*) \in \mathbf{M}(-g_1 , \ldots, -g_m).$
\end{corollary}
\begin{proof}
The assumption (iii) implies that the set $K$ is compact, and so $\mathrm{argmin}_{x \in K} f(x)\ne \emptyset.$ The assumption (i) guarantees that there exists $\lambda^* \in {\Bbb R}^m_+$ such that $(x^*, \lambda^*)$ is a saddle-point of the Lagrangian function $L.$ Let $c \in {\Bbb N}$ be an arbitrary natural number satisfying $c > f(x^*).$ Thanks to Theorem 2.1, we get $f- f(x^*) \in \mathbf{M}( -g_1 , \ldots, -g_m,c-f).$

On the other hand, by taking $c$ large enough, if necessary, from the assumption~ (iii) we may assume that $c - f \in \mathbf{M}(-g_1 , \ldots, -g_m).$ Therefore $f - f(x^*) \in \mathbf{M}( -g_1 , \ldots, -g_m,c-f) = \mathbf{M}(-g_1 , \ldots, -g_m),$ which completes the proof.
\end{proof}

\begin{remark}[{\bf Comparisons with known recent results}]  {\rm
In the special case where the Hessian $\nabla^2f$ of the objective function $f$  is positive definite
at a minimizer $x^* \in {\rm argmin}_Kf$, then the Slater condition ensures that there exists $\lambda^* \in \mathbb{R}^m_+$ such that $(x^*, \lambda^*)$
is a saddle-point of the Lagrangian function $L$, and so, the Hessian $\nabla^2_{xx} L$ of
$L$ is positive definite at $(x^*,\lambda^*)$. Hence, it is easy to see that the above corollary
extends the representation results for convex polynomial optimization established in \cite[Theorem 3.4]{Lasserre2009-1} and \cite[Corollary 3.3]{Klerk2011}.
}\end{remark}

The following simple one dimensional example
illustrates that our representation result can be applied to the case where the Hessian $\nabla^2f$ of the
objective function $f$  is not positive definite at a minimizer.
\begin{example} {\bf (Verifying representation: Non-positive definiteness case of the Hessian $\nabla^2f$)}
{\rm Let $f(x)=x$ and $g(x)=x^2-1$. Then, $K:=\{x \in \mathbb{R} \ | \ g(x) \le 0\}=[-1,1]$. Clearly, ${\rm argmin}_Kf=\{-1\}$ and $f$ is not positive definite at the unique
minimizer $x^*:=-1$. On the other hand, direct verification shows that
$(x^*,\lambda^*):=(-1,\frac{1}{2})$ is a saddle point of the Lagrangian function
$L(x,\lambda) := f(x)+\lambda g(x)=x+\lambda (x^2-1)$, and
$\nabla^2_{xx} L(x^*,\lambda^*) \succ 0.$ Moreover, Slater condition is satisfied and the quadratic module
$\mathbf{M}(-g)$ is Archimedean. So, it follows from the previous corollary that
$f-f(x^*)=f+1 \in \mathbf{M}(-g)$. Indeed,  $f-f(x^*) = x+1 = \frac{1}{2}(x + 1)^2 + \frac{1}{2}(1-x^2) \in \mathbf{M}(-g)$.
}\end{example}


As we see in the following theorem, under the Slater condition and the positive definiteness of the
Hessian of $f$ at a minimizer, we obtain a sharper representation than the one in Theorem \ref{TheoremRepresentation41}.

\begin{theorem} [{\bf Sharp Representation with positive definite $\nabla^2f(x^*)$}] \label{CorollaryRepresentation4}
Let $f$ and $g_1, \ldots, g_m \in \mathbb{R}[\underline{x}]$ be convex polynomials with
$K := \{x\in \mathbb{R}^n \ | \  g_i(x) \le 0, i = 1, \ldots, m\} \ne \emptyset.$ Let
$\mathrm{argmin}_{x \in K} f(x)\ne \emptyset$ and $x^*\in \mathrm{argmin}_{x \in K} f(x)$. If there
exists $x^0 \in {\Bbb R}^n \textrm{ such that } g_i(x^0) < 0,
\textrm{ for } i = 1, \ldots, m$ and if $\nabla^2 f(x^*)\succ 0 $ then, for any $c > f(x^*)$, there exist sum-of-squares polynomials $\sigma_0,\sigma_1 \in \Sigma^2$ and Lagrange multipliers $\lambda_i^* \ge 0$, $i=1,2,\ldots, m$
such that $$f - f(x^*)  =   \sigma_0 - \sum_{i=1}^m \lambda_i^* g_i + \sigma_1(c - f).$$
\end{theorem}

\begin{proof} The Slater condition and convex programming duality guarantee that there exists $\lambda^*\in \mathbb{R}_+^m$ such that
$(x^*, \lambda^*)$ is a saddle-point of the Lagrangian function
$L(x,\lambda) := f(x) + \sum_{i = 1}^m\lambda_ig_i(x)$. So,
for each $x\in \mathbb{R}^n$, $L(x, \lambda^*)\ge L(x^*, \lambda^*)= f(x^*)$. Then,
$$h(x) := L(x, \lambda^*) - f(x^*) = f(x) - f(x^*)  + \sum_{i=1}^m \lambda^*_i g_i(x) \ge 0,
\qquad \forall x \in {\Bbb R}^n.$$

Now, by the assumption, $\nabla^2 f(x^*)\succ 0 $ and so, Lemma~ \ref{HessianLemma} shows that
$f$ is a strictly convex and coercive polynomial. Then, the convex set
$\bar{S} : =  \{x \in {\Bbb R}^n \ | \ f(x) \le c  \}$
is nonempty and compact. Since $h\ge 0$  on $\bar{S}$,  \cite[Corollary 3.6]{Scheiderer2005}
(see also \cite[Example 3.18]{Scheiderer2003}) gives us that
there exist sum-of-squares polynomials $\sigma_0,\sigma_1 \in \Sigma^2$ such that, for each $x\in \mathbb{R}^n$,
\begin{eqnarray*}
h(x) &=&  \sigma_0(x) + \sigma_1(x) (c -f(x)).
\end{eqnarray*}
This reduces to, for each $x\in \mathbb{R}^n$,
\begin{eqnarray*}
f(x) - f(x^*) & = &  \sigma_0 - \sum_{i=1}^m \lambda_i^* g_i(x) + \sigma_1(c - f(x)).
\end{eqnarray*}
Then the conclusion follows.
\end{proof}	

\begin{remark}[{\bf Constraint qualifications}]{\rm
In Corollary 2.1 and Theorem 2.2, we have used the Slater condition for guaranteeing the existence of a saddle-point.
For other general constraint qualifications
ensuring the existence of a saddle point of the Lagrangian function, see \cite{jeya1,jgd}.
}\end{remark}

%
We now show that the Lasserre hierarchy of SDP relaxations of Problem (\ref{OptimizationProblem}) has {\em finite convergence} which means that $f^*_k = f^*$ for some integer $k$ and Problem (\ref{SOSRelaxationProblem1})  achieves its optimal value $f^*_k$.

\begin{theorem}[{\bf Finite Convergence}] \label{FiniteConvergenceTheorem}
For Problem~{\rm (\ref{OptimizationProblem})}, let $L\colon \mathbb{R}^n \times \mathbb{R}^m_+ \rightarrow \mathbb{R}$ be the Lagrangian function defined by
$L(x,\lambda) =f(x)+\sum_{i = 1}^m\lambda_ig_i(x)$. Assume that the Lagrangian function $L$ has a saddle-point $(x^*, \lambda^*)\in K\times \mathbb{R}^m_+$ with $\nabla^2_{xx} L(x^*,\lambda^*) \succ 0$.
Then there exists an integer $k$ such that $f^*_k = f^*$ and Problem {\rm (\ref{SOSRelaxationProblem1})}  achieves its optimal value.
\end{theorem}
\begin{proof}
We know that $f^*_{k} \le f^*$ for all $k \ge 1.$ On the other hand, it follows from Theorem~ \ref{TheoremRepresentation41} that there exist sum-of-squares polynomials
 $\sigma,\sigma_0, \sigma_1,\ldots,\sigma_m \in \Sigma^2$ such that
\begin{eqnarray*}
f  - f^* &=& \sigma_0  - \sigma_1 g_1 - \cdots - \sigma_m g_m+\sigma (c - f).
\end{eqnarray*}
Hence $f^*_k = f^*$ for some $k \in {\Bbb N}.$ As $(x^*, \lambda^*)\in K \times \mathbb{R}^m_+$ is a saddle-point of $L$, $x^*$ is  minimizer of Problem (\ref{OptimizationProblem}) and $f^*=f(x^*)$ which is also a solution of Problem (\ref{SOSRelaxationProblem1}).
\end{proof}

\begin{remark}{\rm
 It is worth noting that in the case, where $\nabla^2 f(x^*)$ is positive definite at a minimizer $x^*$ of
Problem~{\rm (\ref{OptimizationProblem})}, using Theorem 2.2, one can establish finite convergence of a sharper form of approximation Problem
{\rm (\ref{SOSRelaxationProblem1})}, where the Lagrange multipliers, $\lambda_i^*$, $i=1,2,\ldots, m$,
associated with the minimizer $x^*$, are replaced by
$\sigma_i$, $i=1,2,\ldots, m$, in $\mathbf{M}_k$}.
\end{remark}


The following example shows that the finite convergence in the preceding theorem may fail if the saddle-point condition does not hold at a minimizer.
\begin{example}[{\bf Importance of Saddle-point Condition for Finite Convergence}]{\rm  Consider the minimization problem
\begin{equation} \label{CounterExample1}
\min\{f(x,y) \ | \  g(x,y) \le 0 \ \},
\end{equation}
where $f(x,y)= x^2+y^2+x+y$, $g(x,y) = x^2+y^2$
and $K :=\{(x,y) \in \mathbb{R}^2 \ | \ g(x,y) \le 0\}.$

Clearly, the unique minimizer of (\ref{CounterExample1}) is $(x^*,y^*)=(0,0)$, $f^*:=f(x^*,y^*)=0$ and $\nabla^2 f(x^*,y^*) ={\rm diag}(2,2) \succ 0$. It is easy to check that the saddle-point condition is not satisfied at $(x^*,y^*)=(0,0)$.

Now, let $c$ be a real number such that $c > f^*=0$. For each $k$, the $k$th-order relaxation problem of (\ref{CounterExample1}) is
\[
\sup\{\mu \in {\Bbb R} \ | \ f - \mu \in \mathbf{M}_{k}  \},
\]
where $\mathbf{M}_k := \{\sigma_0 -\sigma_1 g+\sigma (c-f) \ | \  \sigma,\sigma_0,\sigma_1 \in \Sigma^2, \deg \sigma_0 \le 2k, \deg \sigma_1 g \le 2k, \deg \sigma (c-f) \le 2k\}$. We now show that the finite convergence fails. We establish this by the method of contradiction. Suppose that Problem (\ref{CounterExample1}) has finite convergence. Then, there exists $k_0 \in \mathbb{N}$, $\sigma,\sigma_0,\sigma_1 \in \Sigma^2$ with ${\rm deg} \, \sigma_0 \le 2k_0, {\rm deg} \, \sigma_1 g \le 2k_0$ and
${\rm deg} \, \sigma (c-f) \le 2k_0$
such that
$f=f - f^*=\sigma_0-\sigma_1 g+\sigma(c-f).$
This gives us, for each $(x,y) \in \mathbb{R}^2,$  that
\begin{equation} \label{eq:00}
\big(1+\sigma(x,y)+\sigma_1(x,y)\big) (x^2+y^2) + (1+\sigma(x,y))(x+y)=\sigma_0(x,y)+c \sigma (x,y) \ge 0.
\end{equation}
 Letting $(x,y)=(-\frac{1}{k},-\frac{1}{k})$ in (\ref{eq:00}), where $k \in \mathbb{N}$, yields
\[
\big(1+\sigma(-\frac{1}{k},-\frac{1}{k})+\sigma_1(-\frac{1}{k},
-\frac{1}{k})\big) \frac{2}{k^2} - (1+\sigma_1(-\frac{1}{k},-\frac{1}{k}))\frac{2}{k} \ge 0.
\]
Then,
\[
 1+\sigma(-\frac{1}{k},-\frac{1}{k})+\sigma_1(-\frac{1}{k},-\frac{1}{k}) \ge k(1+\sigma_1(-\frac{1}{k},-\frac{1}{k})) \ge k,
\]
which is impossible as the left hand side converges to $ 1+\sigma(0,0)+\sigma_1(0,0)$.
}\end{example}

The following theorem shows that the existence of saddle-point of the Lagrangian function of Problem~ (\ref{OptimizationProblem}) at each minimizer is indeed necessary for our finite convergence.

\begin{theorem}[{\bf Necessity of Saddle-point for Finite Convergence}]
For Problem~ {\rm (\ref{OptimizationProblem})}, let $L \colon \mathbb{R}^n \times \mathbb{R}^m_+ \rightarrow \mathbb{R}$ be the Lagrangian function defined by $L(x, \lambda) := f(x)+\sum_{i = 1}^m\lambda_ig_i(x).$
If the Lasserre hierarchy has finite convergence then, for every minimizer $x^*$ of
Problem {\rm (\ref{OptimizationProblem})}, there exists $\lambda^*\in \mathbb{R}_+^m$ such that
$(x^*, \lambda^*)$ is a saddle-point of $L$.
\end{theorem}
\begin{proof}
Assume that the Lasserre hierarchy has finite convergence. Let $x^*\in K$ with $f^*:=f(x^*)= \min_{x \in K} f(x).$ Then,
\[
f  - f^* \ =  \sigma_0  - \sigma_1 g_1 - \cdots - \sigma_m g_m+ \sigma (c - f),
\]
where $\sigma,\sigma_0, \sigma_1,\ldots,\sigma_m \in \Sigma^2$ are sum-of-squares polynomials and $c > f(x^*)$. This gives us that
\[
(1+ \sigma)(f- f^*) =\sigma_0 - \sigma_1 g_1 - \cdots - \sigma_m g_m+ \sigma (c - f^*).
\]
Thus, for all $x \in \mathbb{R}^n$, $$
\big(1+ \sigma(x)\big)(f(x)- f^*) + \sum_{i=1}^m\sigma_i(x)g_i(x) = \overline{\sigma}_0(x),$$
where $\overline{\sigma}_0 := \sigma_0 + \sigma (c-f^*)\in \Sigma^2.$ Let $x=x^*$. Then, we have $\sum_{i=1}^m\sigma_i(x^*)g_i(x^*) = \overline{\sigma}_0(x^*) \ge 0$.
 This together with $\sigma_i \ge 0$ and $x^* \in K$ implies that
$$\overline{\sigma}_0(x^*) = 0 \quad \textrm{ and } \quad \sigma_i(x^*)g_i(x^*)=0, \ i=1,\ldots, m,$$
and hence $\sigma_0(x^*) = \sigma(x^*) = 0.$
 As $\overline{\sigma}_0(x) \ge 0$ and $\overline{\sigma}_0(x^*)=0,$ $x^*$ is a minimizer of $\overline{\sigma}_0$ and so,
\begin{eqnarray*}
 0 \ = \ \nabla \overline{\sigma}_0(x^*) & = & (1+\sigma (x^*))\nabla f(x^*) + \sum_{i=1}^m \sigma_i(x^*) \nabla g_i(x^*) + \sum_{i=1}^m \nabla \sigma_i(x^*)  g_i(x^*) \\
& = & \nabla f(x^*) + \sum_{i=1}^m \sigma_i(x^*) \nabla g_i(x^*) + \sum_{i=1}^m \nabla \sigma_i(x^*)  g_i(x^*).
\end{eqnarray*}
Since $\sigma_i(x^*)g_i(x^*)=0, \ i=1,\ldots, m$, it follows that if $g_i(x^*) < 0$ then $\sigma_i(x^*)=0$ and hence $\nabla \sigma_i(x^*)=0.$ Consequently, $\nabla \sigma_i(x^*)  g_i(x^*)=0$, $i=1,\ldots,m$.  So, we have
\[
\nabla f(x^*) + \sum_{i=1}^m \lambda_i^* \nabla g_i(x^*) =0,
\]
where $\lambda_i^* := \sigma_i(x^*)$, $i=1,\ldots,m$.
Hence, by convexity of $f+ \sum_{i=1}^m\lambda_i^*g_i,$ we get that, for each $x\in \mathbb{R}^n$,
$$f(x)+ \sum_{i=1}^m\lambda_i^*g_i(x) \ge f(x^*)+ \sum_{i=1}^m\lambda_i^*g_i(x^*).$$
It is now easy to check that $(x^*, \lambda^*)$ is a saddle-point of the Lagrangian function of Problem~(\ref{OptimizationProblem}).
\end{proof}

\begin{remark}
{\rm For related necessary conditions for finite convergence of Lasserre hierarchy for optimization problems, where feasible sets are
compact, see \cite{nie2012}}.
\end{remark}

\section{Appendix: Proofs of Coercivity \& Strict Convexity of Convex Polynomials }




\noindent{\bf Proof of Lemma 2.2}.
Let
$$E_h:=\{d \in \mathbb{R}^n \ | \ h(x+td)=h(x), \, \forall \, t \in \mathbb{R} \mbox{ and } \forall \, x \in \mathbb{R}^n\}.$$ Then, it is
easy to verify directly that $E_h$ is a subspace of $\mathbb{R}^n.$ Let $l := n - \dim E_h,$ and let
$e_1,\ldots,e_n \in \mathbb{R}^n$ be an orthonormal basis such that ${\rm span}\{e_{l+1},\ldots,e_n\}=E_h$ and ${\rm span}\{e_1,\ldots,e_l\}=E_h^{\bot}$, where $E_h^{\bot}$ is the orthogonal complement of $E_h$. Let $A := [e_1, \ldots,e_n]$. Then,  $A$ is an orthogonal matrix. Define $g \colon \mathbb{R}^l \rightarrow \mathbb{R}$ by $g(x_1,\ldots,x_l) := h\left(\sum_{i=1}^lx_ie_{i} \right).$ Then, $g$ is a convex polynomial and bounded below on ${\Bbb R}^l.$ Further, we have, for all $x \in \mathbb{R}^n,$
$$h(Ax) = h \left(\sum_{i=1}^n x_ie_i \right) = h \left(\sum_{i=1}^l x_ie_i+\sum_{i=l+1}^n x_ie_i \right) = h\left(\sum_{i=1}^l x_ie_i \right) = g(x_1,\ldots,x_l),$$
where the third equality follows by the fact that $\sum_{i=l+1}^n x_ie_i \in E_h.$

To verify that $g$ is indeed coercive, we assume, on the contrary, that $S:=\{x:g(x) \le \alpha\}$ is
unbounded for some $\alpha \in \mathbb{R}$. Let $\{a_k\} \subseteq S$ such that
$\|a_k\| \rightarrow +\infty$ as $k \rightarrow \infty$. Let $a \in \mathbb{R}^l$. Then, by passing to subsequence if necessary, we may
 assume that $\frac{a_k-a}{\|a_k-a\|} \rightarrow v \neq 0$. Let $t \ge 0.$
For sufficiently large $k$, we have $0 < \frac{t}{\| a_k - a \|} < 1,$ and so
\begin{eqnarray*}
g \left ( a + t \frac{a_k - a}{\| a_k - a \|} \right)
&=& g \left ( \left( 1- \frac{t}{\| a_k - a \|} \right) a + \frac{t}{\| a_k - a \|} a_k  \right)  \\
&\le& \left (1- \frac{t}{\| a_k - a \|} \right) g(a)  + \frac{t}{\| a_k - a \|} g(a^k)  \\
&\le& \max\{g(a),\alpha\}.
\end{eqnarray*}  Letting $k \rightarrow \infty$, we get that $g(a+tv) \le \max\{g(a),\alpha\}$ for all $t \ge 0$.
By assumption, $g$ is bounded below. So, $t \mapsto g(a+tv)$ is either a constant or a polynomial with
even degree $\ge 2$. It then follows that $g$ takes a constant value on $\{a+t v:t \ge 0\}$ for all
$a \in \mathbb{R}^l$.  Then, for all $t \ge 0$ and for any $a \in \mathbb{R}^l$, $g(a-tv)=g(a-tv+tv)=g(a)$.
 Thus,  \begin{equation}\label{eq:pp} g(a)=g(a+tv) \mbox{ for all } a \in \mathbb{R}^l \mbox{ and }
t \in \mathbb{R}.\end{equation}

Let $\tilde{v} := (v^T, 0,\ldots,0)^T \in \mathbb{R}^n$ and $d := A \tilde{v} = \sum_{i = 1}^l v_i e_i \in E_h^\perp.$ Since $v \ne 0,$ $d \ne 0.$ Moreover, for all $x \in \mathbb{R}^n$ and $t \in \mathbb{R},$
$$h(x+td)=h(A(A^{-1}x+t\tilde{v}))=g(z+tv)=g(z)=h(x),$$
where $z=\big((A^{-1}x)_1,\ldots,(A^{-1}x)_l\big) \in \mathbb{R}^l$. So, by definition,  $d \in E_h.$ Consequently, we obtain that
$d \in (E_h \cap E_h^{\bot}) \backslash \{0\},$ which is impossible. Hence, $g$ is coercive.

Since the polynomial $g$ is coercive, there exists ${z^*} := (z_1^*, \ldots, z_l^*) \in {\Bbb R}^l$ such that
$g({z^*}) = \inf_{z \in {\Bbb R}^l} g(z).$ Let
$x^* := A {z^* \choose 0} = z_1^* e_1 + \cdots + z_l^* e_l \in E_h^\perp \subset {\Bbb R}^n.$
Then, $h({x^*}) = \inf_{x \in {\Bbb R}^n} h(x) = g({z}^*).$

\bigskip

\noindent{\bf Proof of Lemma 3.1}.
${\bf (Coercivity)}$ Let $c$ be a real number such that $c \ge f(x_0).$ To prove coercivity of $f$ on $\mathbb{R}^n,$ it suffices to show that the set
$$S := \{x \in {\Bbb R}^n \ | \ f(x) \le c\}$$
is compact. On the contrary, suppose that there exists a sequence $\{a_k\}_{k\ge 0} \subset S$ such that
$\|a_k\| \to \infty$ as $k \to \infty.$ Without lost of generality, we may assume that there exists $v\neq 0$
such that
$$v := \lim_{k \to \infty}\frac{a_k - x_0}{\| a_k - x_0 \|}.$$
Let $t \ge 0.$ For sufficiently large $k$, we have $0 < \frac{t}{\| a_k - x_0 \|} < 1,$ and so
\begin{eqnarray*}
f \left ( x_0 + t \frac{a_k - x_0}{\| a_k - x_0 \|} \right)
&=& f \left ( \left( 1- \frac{t}{\| a_k - x_0 \|} \right) x_0 + \frac{t}{\| a_k - x_0 \|} a_k  \right)  \\
&\le& \left (1- \frac{t}{\| a_k - x_0 \|} \right) f(x_0)  + \frac{t}{\| a_k - x_0 \|} f(a_k)  \\
&\le& c.
\end{eqnarray*}
Letting $k \to \infty,$ we get
$$f(x_0 + tv) \le c, \quad \textrm{ for all }\quad t \ge 0.$$

On the other hand, as the Hessian $\nabla^2 f(x_0)$ is positive definite, $\langle \nabla^2 f(x_0)v, v \rangle  > 0$
and so, for each $t \in {\Bbb R},$
$$f(x_0 + tv) \ = \ f(x_0) + \langle \nabla f(x_0), v \rangle t + \frac{1}{2} \langle \nabla^2 f(x_0)v, v \rangle t^2 + \textrm{ higher order terms in } t.$$
 Hence,  the one dimensional convex polynomial $t \mapsto f(x_0 + tv)$ is of even degree $\ge 2.$ This is a contradiction since $f(x_0 + tv) \le c \quad \textrm{ for all }\quad t \ge 0.$

\noindent ${\bf (Strict \ Convexity)}$ We establish strict convexity of $f$ by the method of contradiction and suppose that $f$ is not strictly convex. Then, there exist $x, y \in {\Bbb R}^n, x \ne y,$ and $t_0 \in (0,1)$ such that
$$f((1 - t_0)  x + t_0 y) = (1 - t_0) f(x) + t_0f(y).$$
Define $h:[0, 1] \to {\Bbb R}$ by $h(t)=f((1 - t)  x + t y) - (1 - t) f(x) - tf(y)$. Then, $h$ is a convex polynomial, $h(t) \le 0$, for each $t\in [0, 1]$ and $h(t_0)=0=\max_{t \in [0,1]} h(t)$. As $h$ is a convex function on $[0,1]$, it attains its maximum on the extreme points of $[0, 1]$, and so,
$$f((1 - t)  x + t y) = (1 - t) f(x) + tf(y), \ \ \forall t\in [0,1 ].$$

Now, define a polynomial $\varphi$ on ${\Bbb R}$ by $\varphi(\lambda) := f\big(x+\lambda(y-x)\big)$, $\lambda \in {\Bbb R}.$ Clearly, $\varphi$ is affine on $[0,1],$ and moreover, it is coercive on $\mathbb{R}$  because $f$ is coercive on $\mathbb{R}^n$, shown above. We show that $\varphi$ is indeed affine over $\mathbb{R}.$ Let the degree of the one-dimensional polynomial $\varphi$ be $d$. Then, for each $\lambda\in {\Bbb R}$,
$$\varphi(\lambda) = \varphi(0) + \varphi'(0) \lambda  + \frac{\varphi''(0)}{2}\lambda^2 + \cdots + \frac{\varphi^{(d)}(0)}{d!} \lambda^d.$$
As $\varphi$ is affine over $[0,1]$, $\varphi^{(i)} (0) = 0$ for $i = 2, \ldots, d$, and so,
$\varphi(\lambda) = \varphi(0) + \varphi'(0)\lambda$. Hence,  $\varphi$ is affine over ${\Bbb R}$. This contradicts the fact that $\varphi$ is coercive on ${\Bbb R}.$
\medskip

\noindent\textbf{Remark 4.1}. The conclusion of Lemma 2.1 may also be derived from error bound results of
convex polynomials (see e.g  \cite{yang} and other references therein). However, for the sake of simplicity and self-containment, we have
given an elementary direct proof for Lemma 2.1.

\end{document}